\documentclass[oneside,leqno,10pt]{article}
\usepackage{amssymb,amsmath,latexsym,amsthm,stmaryrd}

\def\ga{\mathfrak{a}}
\def\gb{\mathfrak{b}}

\def\gg{\mathfrak{g}}
\def\gh{\mathfrak{h}}

\def\gk{\mathfrak{k}}
\def\gl{\mathfrak{l}}
\def\gm{\mathfrak{m}}
\def\gn{\mathfrak{n}}

\def\gp{\mathfrak{p}}

\def\gs{\mathfrak{s}}
\def\gt{\mathfrak{t}}
\def\gu{\mathfrak{u}}
\def\gv{\mathfrak{v}}

\def\gz{\mathfrak{z}}

\def\C{\mathbb{C}}

\def\R{\mathbb{R}}

\def\Z{\mathbb{Z}}

\def\cC{\mathcal{C}}

\def\cE{\mathcal{E}}
\def\cF{\mathcal{F}}

\def\cH{\mathcal{H}}

\def\cL{\mathcal{L}}

\def\cO{\mathcal{O}}

\def\cU{\mathcal{U}}

\def\Ad{{\rm Ad}\,}
\def\ad{{\rm ad}\,}
\def\Pf{{\rm Pf}\,}

\def\Ind{{\rm Ind\,}}

\def\tr{{\rm trace\,}}

\def\Det{\rm Det}

\newtheorem{theorem}[equation]{Theorem}

\newtheorem{lemma}[equation]{Lemma}

\newtheorem{corollary}[equation]{Corollary}

\newtheorem{proposition}[equation]{Proposition}

\newtheorem{definition}[equation]{Definition}

\newtheorem{remark}[equation]{Remark}

\def\sideremark#1{\ifvmode\leavevmode\fi\vadjust{\vbox to0pt{\vss
 \hbox to 0pt{\hskip\hsize\hskip1em
\vbox{\hsize2cm\tiny\raggedright\pretolerance10000 
 \noindent #1\hfill}\hss}\vbox to8pt{\vfil}\vss}}} 

\title{The Plancherel Formula for Minimal Parabolic Subgroups}

\author{Joseph A. Wolf}
\date{December 17, 2013}

\begin{document}

\maketitle

\abstract{In a recent paper we found conditions for a nilpotent Lie
group to be foliated
into subgroups that have square integrable 
unitary representations that fit together to form a filtration by normal
subgroups.  That resulted in explicit character formulae, Plancherel
Formulae and multiplicity formulae.  We also showed that nilradicals $N$ of
minimal parabolic subgroups $P = MAN$ enjoy that ``stepwise square integrable''
property.  Here we extend those results from $N$ to $P$.  The Pfaffian
polynomials, which give orthogonality relations and Plancherel density
for $N$, also give a semi-invariant differential operator that compensates lack
of unimodularity for $P$.  The result is a completely explicit Plancherel
Formula for $P$.}

\section{Introduction}
\label{sec1}
A connected simply connected Lie group $N$
with center $Z$ is called {\em square integrable} if it has unitary
representations $\pi$ whose coefficients $f_{u,v}(x) = 
\langle u, \pi(x)v\rangle$ satisfy $|f_{u,v}| \in \cL^2(N/Z)$.  
C.C. Moore and the author worked out the structure and representation
theory of these groups \cite{MW1973}.  If $N$ has one 
such square integrable representation then there is a certain polynomial
function $\Pf(\lambda)$ on the linear dual space $\gz^*$ of the Lie algebra of
$Z$ that is key to harmonic analysis on $N$.  Here $\Pf(\lambda)$ is the
Pfaffian of the antisymmetric bilinear form on $\gn / \gz$ given by
$b_\lambda(x,y) = \lambda([x,y])$.  The square integrable
representations of $N$ are the 
$\pi_\lambda$ where $\lambda \in \gz^*$ with $\Pf(\lambda) \ne 0$,
Plancherel almost all irreducible unitary representations of $N$ are square
integrable, and up to an explicit constant 
$|\Pf(\lambda)|$ is the Plancherel density of the unitary
dual $\widehat{N}$ at $\pi_\lambda$.  
This theory has proved to have serious analytic consequences.  For example,
for most commutative nilmanifolds $G/K$, i.e. Gelfand pairs $(G,K)$ 
where a nilpotent subgroup $N$ of $G$ acts transitively on $G/K$, the
group $N$ has square integrable representations \cite{W2007}.
And it is known just which maximal parabolic subgroups of semisimple Lie groups
have square integrable nilradical \cite{W1979}.
\medskip

In \cite{W2012} and \cite{W2013} the theory of square integrable nilpotent
groups was extended to ``stepwise
square integrable'' nilpotent groups.
By definition they are the connected simply connected nilpotent Lie groups
that satisfy (\ref{setup}) just below.  We use $L$ and $\gl$ to avoid 
conflict of
notation with the $M$ and $\gm$ of minimal parabolic subgroups.
$Z_r$ denotes the center of $L_r$ and $\gv_r$ is a vector space complement
to $\gz_r$ in $\gl_r$.
\begin{equation}\label{setup}
\begin{aligned}
N &= L_1L_2\dots L_{m-1}L_m \text{ where }\\
 &\text{(a) each $L_r$ has unitary representations with coefficients in
$\cL^2(L_r/Z_r)$,} \\
 &\text{(b) each } N_r := L_1L_2\dots L_r \text{ is a normal subgroup of } N
   \text{ with } N_r = N_{r-1}\rtimes L_r \text{ semidirect,}\\
 &\text{(c) decompose }\gl_r = \gz_r + \gv_r \text{ and } \gn = \gs + \gv
        \text{ as vector direct sums where } \\
 &\phantom{XXXX}\gs = \oplus\, \gz_r \text{ and } \gv = \oplus\, \gv_r;
    \text{ then } [\gl_r,\gz_s] = 0 \text{ and } [\gl_r,\gl_s] \subset \gv
        \text{ for } r > s\,.
\end{aligned}
\end{equation}
The choice of the $\gv_r$ is not important in (\ref{setup}), as long
as $[\gl_r,\gl_s] \subset \gv \text{ for } r > s$\,, because 
integration and Lie brackets in $\gl_r$ are really over $\gl_r/\gz_r$
rather than $\gv_r$\,.  Denote
\begin{equation}\label{c-d}
\begin{aligned}
&\text{(a) }d_r = \tfrac{1}{2}\dim(\gl_r/\gz_r) \text{ so }
        \tfrac{1}{2} \dim(\gn/\gs) = d_1 + \dots + d_m\,,
        \text{ and } c = 2^{d_1 + \dots + d_m} d_1! d_2! \dots d_m!\\
&\text{(b) }b_{\lambda_r}: (x,y) \mapsto \lambda([x,y])
        \text{ viewed as a bilinear form on } \gl_r/\gz_r \\
&\text{(c) }S = Z_1Z_2\dots Z_m = Z_1 \times \dots \times Z_m \text{ where } Z_r
        \text{ is the center of } L_r \\
&\text{(d) }\Pf: \text{ polynomial } \Pf(\lambda) = \Pf_{\gl_1}(b_{\lambda_1})
        \Pf_{\gl_2}(b_{\lambda_2})\dots \Pf_{\gl_m}(b_{\lambda_m}) \text{ on } 
	\gs^* \\
&\text{(e) }\gt^* = \{\lambda \in \gs^* \mid \Pf(\lambda) \ne 0\} \\
&\text{(f) } \pi_\lambda \in \widehat{N} \text{ where } \lambda \in \gt^*:
    \text{ irreducible unitary rep. of } N = L_1L_2\dots L_m
\end{aligned}
\end{equation}

We recall the Schwartz space $\cC(N)$, following the lines of the exposition in 
\cite[Section 1]{C1989}.  Start with a norm on $N$. For example the operator
norm $||\alpha(x)||$, where $\alpha$ is a faithful finite dimensional 
representation of $N$ by unipotent linear transformations of a Hilbert space,  
defines a norm $|x| = \sup(||\alpha(x)|| , ||\alpha(x^{-1})||)$.
Or one can use $|x| = (1 + distance(x,1)^2)$ 
with a left invariant Riemannian metric on $N$.
The properties we need are that the norm be continuous and satisfy 
(i) $|1| = 1$, (ii) $|x| \geqq 1$, (iii) $|x^{-1}| = |x|$ and
(iv) $|x|\cdot |y|^{-1} \leqq |xy| \leqq |x|\cdot |y|$.
Write $\ell$ for the left action of the 
universal enveloping algebra $\cU(\gn)$ on $C^\infty(N)$ and 
$r$ for the right action.  
The {\sl Schwartz space} $\cC(N)$, also called the 
space of rapidly decreasing smooth functions on $N$, consists of all 
$f\in C^\infty(N)$ such that 
$$
\nu_{a,k}(f) :=
\sup_{x \in N} |x|^k |\ell(a)(f)(x)| < \infty \text{ for every }
	a \in \cU(\gn) \text{ and every integer } k \geqq 0.
$$
The seminorms $\nu_{a,k}$ define a nuclear Fr\' echet space topology on
$\cC(N)$ and we have continuous inclusions
$C^\infty(N) \hookrightarrow \cC(N) \hookrightarrow \cL^2(N)$
with dense images.  
Two continuous norms that satisfy our conditions (i) through (iv) 
are equivalent (each bounded by a multiple of the other), 
so they give the same Schwartz space.
If $f \in \cC(N)$ then $\ell(a)\,r(b)\,(f) \in \cC(N) \subset \cL^2(N)$ 
for all $a,b \in \cU(\gn)$.
Since $N$ is connected, simply connected and nilpotent, the exponential map 
$\exp: \gn \to N$ is polynomial, and $f \in \cC(N)$ if and only if its lift 
$f_1(\xi) = f(\exp(\xi))$ belongs to the classical Schwartz space of the real
vector space $\gn$.  
\medskip

If $\pi \in \widehat{N}$ and
$f \in \cC(N)$ then $\pi(f) := \int_N f(x)\pi(x)dx$ is trace class and
$\Theta_\pi : f \mapsto \tr \pi(f)$ is a tempered distribution
(distribution that extends by continuity from $C_c^\infty$ to $\cC$) on 
$N$ called the {\sl distribution character} of $\pi$.  The point, now,
is that Plancherel measure on $\widehat{N}$ is concentrated on 
$\{\pi_\lambda \mid \lambda \in \gt^*\}$, and

\begin{theorem}\label{plancherel-general}
Let $N$ be a connected simply connected nilpotent Lie group that
satisfies {\rm (\ref{setup})}.  Then Plancherel measure for $N$ is
concentrated on $\{\pi_\lambda \mid \lambda \in \gt^*\}$.
If $\lambda \in \gt^*$, and if $u$ and $v$ belong to the
representation space $\cH_{\pi_\lambda}$ of $\pi_\lambda$,  then
the coefficient $f_{u,v}(x) = \langle u, \pi_\nu(x)v\rangle$
satisfies
\begin{equation}
||f_{u,v}||^2_{\cL^2(N / S)} = \frac{||u||^2||v||^2}{|\Pf(\lambda)|}\,.
\end{equation}
Recall $c = 2^{d_1 + \dots + d_m} d_1! d_2! \dots d_m!$ from $(\ref{c-d}(a))$.
Then the distribution character $\Theta_{\pi_\lambda}$ of $\pi_{\lambda}$ 
satisfies
\begin{equation}
\Theta_{\pi_\lambda}(f) = c^{-1}|\Pf(\lambda)|^{-1}\int_{\cO(\lambda)}
        \widehat{f_1}(\xi)d\nu_\lambda(\xi) \text{ for } f \in \cC(N)
\end{equation}
where $\cC(N)$ is the Schwartz space, $f_1$ is the lift
$f_1(\xi) = f(\exp(\xi))$, $\widehat{f_1}$ is its classical Fourier transform,
$\cO(\lambda)$ is the coadjoint orbit $\Ad^*(N)\lambda = \gv^* + \lambda$,
and $d\nu_\lambda$ is the translate of normalized Lebesgue measure from
$\gv^*$ to $\Ad^*(N)\lambda$.  The Plancherel Formula on $N$ is
\begin{equation}
f(x) = c\int_{\gt^*} \Theta_{\pi_\lambda}(r_xf) |\Pf(\lambda)|d\lambda
        \text{ for } f \in \cC(N).
\end{equation}
\end{theorem}
\begin{definition}\label{stepwise2}
{\rm The representations $\pi_\lambda$ of (\ref{c-d}(f)) are the
{\it stepwise square integrable} representations of $N$ relative to
the decomposition (\ref{setup}).}\hfill $\diamondsuit$
\end{definition}

One of the main results of \cite{W2012} and \cite{W2013} is that nilradicals
of minimal parabolic subgroups are stepwise square integrable.  Even the
simplest case, the case of a minimal parabolic in $SL(n;\R)$, was a big
improvement over earlier results on the group of strictly upper triangular
real matrices.  Here we extend the results of \cite{W2012} and \cite{W2013}
to obtain explicit Plancherel Formulae for the minimal parabolic $P$ itself.
This is done by construction of a Dixmier--Puk\' anszky operator on $\cL^2(P)$,
i.e. a pseudo--differential operator that compensates lack of unimodularity
on $P$.  The Dixmier--Puk\' anszky operator is explicit; it is 
constructed from the Pfaffian polynomials of (\ref{c-d}d).  The construction
gives a beautiful relation between the Dixmier--Puk\' anszky operator of $P$
and the Plancherel density of its nilradical.
\medskip

In Section \ref{sec2} we review the restricted root structure, stepwise
square integrable representations, character formulae and the Plancherel
(or Fourier Inversion) Formula for nilradicals of minimal parabolic
subgroups.  Some of the restricted root results are discussed further
in Section \ref{sec7}, a sort of appendix, 
where we placed them because they add to, but are not needed
for, the main results.
\medskip

Is Section \ref{sec3} we discuss the structure and action of the group
$M$ in a minimal parabolic $P = MAN$.  The notion of principal orbit 
gives a uniform description of the stabilizers of stepwise square
integrable representations of $N$.  We also show triviality of a certain
Mackey obstruction, leading to an explicit Plancherel Formula for $MN$.
\medskip

In Section \ref{sec4} we work out the Dixmier--Puk\' anszky operator of $P$
in terms of the Pfaffian (which gives Plancherel density on $N$) and a
certain explicit ``quasi--central determinant'' polynomial.
\medskip

In Section \ref{sec5} we apply the Mackey machine to give an explicit
description of subsets of $\widehat{P}$ and $\widehat{AN}$ that
carry Plancherel measure.  The point here is that the description
is explicit.
\medskip

Finally in Section \ref{sec6} we give explicit Plancherel Formulae for
the minimal parabolic subgroups $P = MAN$ and their exponential
solvable subgroups $AN$.

\section{Minimal Parabolics: Structure of the Nilradical}
\label{sec2}
\label{iwasawa}
\setcounter{equation}{0}
Let $G$ be a real reductive Lie group.  We recall some structural results
on its minimal parabolic subgroups, some standard and some from \cite{W2013}.
\medskip

Fix an Iwasawa decomposition $G = KAN$.  As usual, write $\gk$ for the Lie 
algebra of $K$, $\ga$ for the Lie algebra of $A$, and $\gn$ for the
Lie algebra of $N$.  Complete $\ga$ to a Cartan subalgebra $\gh$ of $\gg$.
Then $\gh = \gt + \ga$ with $\gt = \gh \cap \gk$.  Now we have root systems
\begin{itemize}
\item $\Delta(\gg_\C,\gh_\C)$: roots of $\gg_\C$ relative to $\gh_\C$ 
(ordinary roots), and

\item $\Delta(\gg,\ga)$: roots of $\gg$ relative to $\ga$ (restricted roots). 

\item $\Delta_0(\gg,\ga) = \{\gamma \in \Delta(\gg,\ga) \mid 
	2\gamma \notin \Delta(\gg,\ga)\}$ (nonmultipliable restricted roots).
\end{itemize}
Sometimes we will identify a restricted root
$\gamma = \alpha|_\ga$, $\alpha \in \Delta(\gg_\C,\gh_\C)$ and 
$\alpha |_\ga \ne 0$, with the set 
\begin{equation}\label{resrootset}
[\gamma] := 
\{\alpha' \in \Delta(\gg_\C,\gh_\C) \mid \alpha'|_\ga = \alpha|_\ga\}
\end{equation}
of all roots that restrict to it.  Further, 
$\Delta(\gg,\ga)$ and $\Delta_0(\gg,\ga)$ are root 
systems in the usual sense.  Any positive system 
$\Delta^+(\gg_\C,\gh_\C) \subset \Delta(\gg_\C,\gh_\C)$ defines positive 
systems
\begin{itemize}
\item $\Delta^+(\gg,\ga) = \{\alpha|_\ga \mid \alpha \in 
\Delta^+(\gg_\C,\gh_\C) 
\text{ and } \alpha|_\ga \ne 0\}$ and $\Delta_0^+(\gg,\ga) =
\Delta_0(\gg,\ga) \cap \Delta^+(\gg,\ga)$.
\end{itemize}
\noindent We can (and do) choose $\Delta^+(\gg,\gh)$ so that 
\begin{itemize}
\item$\gn$ is the sum of the positive restricted root spaces and
\item if $\alpha \in \Delta(\gg_\C,\gh_\C)$ and $\alpha|_\ga \in
\Delta^+(\gg,\ga)$ then $\alpha \in \Delta^+(\gg_\C,\gh_\C)$.
\end{itemize}
\medskip

Two roots are called {\em strongly orthogonal} if their sum and their
difference are not roots.  Then they are orthogonal.  We define
\begin{equation}\label{cascade}
\begin{aligned}
&\beta_1 \in \Delta^+(\gg,\ga) \text{ is a maximal positive restricted root
and }\\
& \beta_{r+1} \in \Delta^+(\gg,\ga) \text{ is a maximum among the roots of }
\Delta^+(\gg,\ga) \text{ orthogonal to all } \beta_i \text{ with } i \leqq r
\end{aligned}
\end{equation}
Then the $\beta_r$ are mutually strongly orthogonal.  This is Kostant's
cascade construction.  Note that each $\beta_r \in \Delta_0^+(\gg,\ga)$.
Also note that $\beta_1$ is unique if and only if $\Delta(\gg,\ga)$ is
irreducible.
\medskip

For $1\leqq r \leqq m$ define 
\begin{equation}\label{layers}
\begin{aligned}
&\Delta^+_1 = \{\alpha \in \Delta^+(\gg,\ga) \mid \beta_1 - \alpha \in \Delta^+(\gg,\ga)\} 
\text{ and }\\
&\Delta^+_{r+1} = \{\alpha \in \Delta^+(\gg,\ga) \setminus (\Delta^+_1 \cup \dots \cup \Delta^+_r)
	\mid \beta_{r+1} - \alpha \in \Delta^+(\gg,\ga)\}.
\end{aligned}
\end{equation} 

\begin{lemma} \label{fill-out} {\rm \cite[Lemma 6.3]{W2013}}
If $\alpha \in \Delta^+(\gg,\ga)$ then either 
$\alpha \in \{\beta_1, \dots , \beta_m\}$
or $\alpha$ belongs to exactly one of the sets $\Delta^+_r$\,.
In particular the Lie algebra $\gn$ of $N$ is the
vector space direct sum of its subspaces
\begin{equation}\label{def-m}
\gl_r = \gg_{\beta_r} + {\sum}_{\Delta^+_r}\, \gg_\alpha 
\text{ for } 1\leqq r\leqq m
\end{equation}
\end{lemma}

\begin{lemma}\label{layers2}{\rm \cite[Lemma 6.4]{W2013}}
The set $\Delta^+_r\cup \{\beta_r\}  
= \{\alpha \in \Delta^+ \mid \alpha \perp \beta_i \text{ for } i < r
\text{ and } \langle \alpha, \beta_r\rangle > 0\}.$
In particular, $[\gl_r,\gl_s] \subset \gl_t$ where $t = \min\{r,s\}$.
Thus $\gn$ has an increasing filtration
by ideals
\begin{equation}\label{def-filtration}
\gn_r = \gl_1 + \gl_2 + \dots + \gl_r \text{ for } 1 \leqq r \leqq m
\end{equation}
with a corresponding group level decomposition by normal subgroups $N_r$ where
\begin{equation}\label{def-filtration-group}
N = L_1L_2\dots L_m \text{ with } N_r = N_{r-1}\rtimes L_r
	\text{ for } 1 \leqq r \leqq m.
\end{equation}
\end{lemma}

The structure of $\Delta^+_r$, and later of $\gl_r$, is exhibited by a 
particular Weyl group element $s_{\beta_r} \in W(\gg,\ga)$ and its
negative.  Specifically,
\begin{equation}\label{beta-reflect}
s_{\beta_r} \text{ is the Weyl group reflection in } \beta_r
\text{ and } \sigma_r: \Delta(\gg,\ga) \to \Delta(\gg,\ga) \text{ by }
\sigma_r(\alpha) = -s_{\beta_r}(\alpha).
\end{equation}
Here $\sigma_r(\beta_s) = -\beta_s$ for $s \ne r$, $+\beta_s$ if $s = r$.
If $\alpha \in \Delta^+_r$ we still have $\sigma_r(\alpha) \perp \beta_i$
for $i < r$ and $\langle \sigma_r(\alpha), \beta_r\rangle > 0$.  If
$\sigma_r(\alpha)$ is negative then $\beta_r - \sigma_r(\alpha) > \beta_r$
contradicting the maximality property of $\beta_r$.  Thus, using 
Lemma \ref{layers2}, $\sigma_r(\Delta^+_r) = \Delta^+_r$.
This divides each $\Delta^+_r$ into pairs:

\begin{lemma} \label{layers-nilpotent}{\rm \cite[Lemma 6.8]{W2013}}
If $\alpha \in \Delta^+_r$ then $\alpha + \sigma_r(\alpha) = \beta_r$.
{\rm (}Of course it is possible that 
$\alpha = \sigma_r(\alpha) = \tfrac{1}{2}\beta_r$ when 
$\tfrac{1}{2}\beta_r$ is a root.{\rm ).}
If $\alpha, \alpha' \in \Delta^+_r$ and $\alpha + \alpha' \in \Delta(\gg,\ga)$
then $\alpha + \alpha' = \beta_r$\,.
\end{lemma}

It comes out of Lemmas \ref{fill-out} and \ref{layers2} that the 
decompositions of (\ref{layers}), (\ref{def-m}) and
(\ref{def-filtration}) satisfy (\ref{setup}), so
Theorem \ref{plancherel-general}
applies to nilradicals of minimal parabolic subgroups.  In other words,

\begin{theorem}\label{iwasawa-layers}{\rm \cite[Theorem 6.16]{W2013}}
Let $G$ be a real reductive Lie group, $G = KAN$ an Iwasawa
decomposition, $\gl_r$ and $\gn_r$ the subalgebras of $\gn$ defined in 
{\rm (\ref{def-m})} and {\rm (\ref{def-filtration})},
and $L_r$ and $N_r$ the corresponding analytic subgroups of $N$.  
Then the $L_r$ and $N_r$ satisfy {\rm (\ref{setup})}.  In particular,
Plancherel measure for $N$ is
concentrated on $\{\pi_\lambda \mid \lambda \in \gt^*\}$.
If $\lambda \in \gt^*$, and if $u$ and $v$ belong to the
representation space $\cH_{\pi_\lambda}$ of $\pi_\lambda$,  then
the coefficient $f_{u,v}(x) = \langle u, \pi_\lambda(x)v\rangle$
satisfies
\begin{equation}
||f_{u,v}||^2_{\cL^2(N / S)} = \frac{||u||^2||v||^2}{|\Pf(\lambda)|}\,.
\end{equation}
The distribution character $\Theta_{\pi_\lambda}$ of $\pi_{\lambda}$ satisfies
\begin{equation}
\Theta_{\pi_\lambda}(f) = c^{-1}|\Pf(\lambda)|^{-1}\int_{\cO(\lambda)}
        \widehat{f_1}(\xi)d\nu_\lambda(\xi) \text{ for } f \in \cC(N)
\end{equation}
where $\cC(N)$ is the Schwartz space, $f_1$ is the lift
$f_1(\xi) = f(\exp(\xi))$, $\widehat{f_1}$ is its classical Fourier transform,
$\cO(\lambda)$ is the coadjoint orbit $\Ad^*(N)\lambda = \gv^* + \lambda$,
and $d\nu_\lambda$ is the translate of normalized Lebesgue measure from
$\gv^*$ to $\Ad^*(N)\lambda$.  The Plancherel Formula on $N$ is
\begin{equation}
f(x) = c\int_{\gt^*} \Theta_{\pi_\lambda}(r_xf) |\Pf(\lambda)|d\lambda
        \text{ for } f \in \cC(N).
\end{equation}
\end{theorem}

\section{Minimal Parabolics: $\mathbf{M}$-Orbit Structure}
\label{sec3}
\setcounter{equation}{0}
Recall the Iwasawa decomposition $G = KAN$  and  the corresponding minimal 
parabolic subgroup $P = MAN$ where $M$ is the
centralizer of $A$ in $K$.  We write ${}^0$ for identity
component, so $P^0 = M^0AN$.  

\begin{lemma} \label{cone}
Recall the $\Pf$--nonsingular 
set $\gt^* = \{\lambda \in \gs^* \mid \Pf(\lambda) \ne 0\}$ of
{\rm (\ref{c-d}e)}.  Then $\Ad^*(M)\gt^* = \gt^*$.  Further, if 
$\lambda \in \gt^*$ and $c \ne 0$ then $c\lambda \in \gt^*$, 
in fact $\Pf(c\lambda) = c^{\dim(\gn/\gs)/2} \Pf(\lambda)$.
\end{lemma}

\begin{proof} All the ingredients in the formula for $\lambda \mapsto
\Pf(\lambda)$ are $\Ad^*(M)$--equivariant, so $\Ad^*(M)\gt^* = \gt^*$.
By definition the bilinear form $b_\lambda$ on $\gn/\gs$
satisfies $b_{c\lambda} = c b_\lambda$, so $\Pf(c\lambda) = c^{\dim(\gn/\gs)/2}
\Pf(\lambda)$.  
\end{proof}

Choose an $M$--invariant inner product $(\mu,\nu)$ on $\gs^*$\,. Denote
$\gs^*_t = \{\lambda \in \gs^* \mid (\lambda,\lambda) = t^2\}$, the sphere of
radius $t$.  Consider the action of $M$ on $\gs^*_t$.  Recall that two orbits
$\Ad^*(M)\mu$ and $\Ad^*(M)\nu$ are of the {\sl same orbit type} if the 
isotropy subgroups $M_\mu$ and $M_\nu$ are conjugate, and an orbit is
{\sl principal} if all nearby orbits are of the same type.  Since $M$ and
$\gs^*_t$ are compact, there are only finitely many orbit types of $M$
on $\gs^*_t$, there is only one principal orbit type, and the union of the
principal orbits forms a dense open subset of $\gs^*_t$ whose complement has
codimension $\geqq 2$.  There are many good expositions of this material,
for example \cite[Chapter 4, Section 3]{B1972} for a complete treatment, 
\cite[Part II, Chapter 3, Section 1]{GOV1993} modulo
references to \cite{B1972}, and \cite[Cap. 5]{N2005} for a more basic
treatment but still with some references to \cite{B1972}.
\medskip

Since the action of $M$ on $\gs^*$ commutes with dilation, the above mentioned 
structural results on the $\gs_t$ also hold on 
$\gs^* = \bigcup_{t \geq 0} \gs^*_t$.  
Define the $\Pf$-nonsingular principal orbit set as follows:
\begin{equation}\label{defregset}
\gu^* = \{\lambda \in \gt^* \mid \Ad^*(M)\lambda \text{ is a principal }
	M\text{-orbit on } \gs^*\}.
\end{equation} 
Summarizing the short discussion,

\begin{lemma}\label{princ-orbit}
The principal orbit set $\gu^*$ is a dense open set with complement 
of codimension $\geqq 2$
in $\gs^*$.  If $\lambda \in \gu^*$ and $c \ne 0$ then
$c\lambda \in \gu^*$ with isotropy $M_{c\lambda} = M_\lambda$\,.
\end{lemma}

Fix $\lambda \in \gu^*_t := \gu^* \cap \gs^*_t$\,, so $\Ad^*(M)\lambda$ 
is a $\Pf$-nonsingular principal orbit of $M$ on the sphere $\gs^*_t$.  Then 
$\Ad^*(M^0)\lambda$ is a principal orbit of $M^0$ on $\gs^*_t$.  
Principal orbit isotropy subgroups of compact connected
linear groups are studied in detail in \cite{HH1970} so the possibilities for 
$(M^0)_\lambda$ are essentially known.  
\medskip

\begin{lemma}\label{m-components}
Suppose that $G$ is connected and linear.  Then
$M = (\exp(i\ga) \cap K)Z_G M^0$ where $Z_G$ is the center of $G$,
and its action on a restricted root space $\gg_\alpha$
has form $\exp(i\alpha(\xi))|_{\gg_\alpha} = \pm 1$.  In particular 
$(\exp(i\ga) \cap K)$ is an elementary abelian $2$-subgroup of 
$M$ that meets each of its topological components.
\end{lemma}

\begin{proof} A Cartan subgroup $B \subset M$ meets every component of $M$.
The complex Cartan $(BA)_\C = \exp(\gb_\C)\exp(\ga_\C) \subset G_\C$ 
is connected, and $\exp(\gb)$ and $\exp(\ga)$ are connected as well, so
the components of $(BA)\cap G$ are given by $\exp(i\gb)\exp(i\ga)\cap G$.
As $\exp(i\gb)$ is split over $\R$ the components of $(BA)\cap G$ are 
given by $\exp(i\ga)\cap G = \exp(i\ga)\cap K$.  The Cartan involution 
$\theta$ of $G$ with fixed point set $K$ fixes every element of 
$K$ and sends every element of $\exp(i\ga)$ to its inverse, so
$\exp(i\ga) \cap K$ is an elementary abelian $2$-group that meets
every component of $M$.  The restricted root spaces $\gg_\alpha$
are joint eigenspaces of $\ga$, so every element of $\exp(i\ga) \cap K$
acts on each $\gg_\alpha$ by a scalar multiplication $\pm 1$.
\end{proof} 

Define $F$ to be the elementary abelian $2$-subgroup $\exp(i\ga) \cap K$ 
of $M$ considered in Lemma \ref{m-components}.
In order to see exactly how $F$ acts on $\gs^*$ we use a result of Kostant
applied to the centralizer of $Z_M(M^0)A$:

\begin{lemma}\label{resroot_irred}{\rm \cite[Theorem 8.13.3]{W2011}}
Suppose that $G$ is connected.  Then the adjoint representation of $M$ 
on $\gg$ preserves each restricted root space, say acting by $\eta_\alpha$ 
on $\gg_\alpha$, and each $\eta_\alpha|_{M^0}$ is irreducible.  
\end{lemma}

Now we have the action of $F$ on $\gs^*$, as follows.

\begin{proposition}\label{m-components-2}
The group $\Ad^*(F)$ acts trivially on $\gs^*$.
\end{proposition}

\begin{proof} Each of the strongly orthogonal roots gives us a 
$\theta$-stable subalgebra $\gg[\beta_r] \cong \gs\gl(2;\R)$ of $\gg$.
It has standard basis $\{x_r, y_r, h_r\}$ where $h_r \in \ga$ and each
$x_r \in \gz_r \subset \gs$.  Now 
$\ga = \ga_\diamondsuit \oplus \bigoplus \sum \R x_r$
where $\ga_\diamondsuit$ (notation to be justified by (\ref{alambda})) is 
the intersection of
the kernels of the $\beta_r$\,.  As defined, $\ad^*(\ga_\diamondsuit)$ vanishes
on $\sum \R x_r$\,.  By strong orthogonality of $\{\beta_r\}$,\ each
$\ad^*(h_s\C)$ is trivial on $\R x_r$ for $s \ne r$.  Further
$\ad(\exp(\C h_r)\cap K)$ is trivial on $\R x_r$ by a glance at $\gs\gl(2;\R)$.
We have shown that $\Ad(F)x_r = x_r$ for each $r$.  Since $M^0$ is
irreducible on each $\gz_r = \gg_{\beta_r}$ by Lemma \ref{resroot_irred}, and
$M$ centralizes $A$, now $\Ad(F)x = x$ for all $x \in \gz_r$ and all $r$.
\end{proof}

Combining Lemma \ref{m-components} and Proposition \ref{m-components-2}, 
the action of $M_\lambda$ is given by the action of the
identity component of $M$:
\begin{lemma}\label{components}
If  $\lambda \in \gt^*$ then its $M$-stabilizer 
$M_\lambda$ is given by $M_\lambda = F\cdot (M^0)_\lambda$\,.
\end{lemma}
In view of Lemma \ref{components}, the group $M_\lambda$ is specified by the
work of W.--C. and W.--Y. Hsiang \cite{HH1970} on the structure and
classification of principal orbits of compact connected linear groups.
\medskip

Fix $\lambda \in \gt^*$\,, so $\pi_\lambda \in \widehat{N}$ is stepwise
square integrable (Definition \ref{stepwise2}).  Consider the semidirect 
product group $N\rtimes M_\lambda$.  We write $\cH_\lambda$ for the 
representation space of $\pi_\lambda$\,.
The next step is to extend the representation $\pi_\lambda$ to
a unitary representation $\pi_\lambda^\dagger$ of $N\rtimes M_\lambda$
on the same representation space $\cH_\lambda$.  By
\cite[Th\' eor\` eme 6.1]{D1972} the Mackey obstruction 
$\varepsilon \in H^2(M_\lambda;U(1))$ to this extension, 
where $U(1) = \{|z| = 1\}$, has order $1$ or $2$.  But here 
the Mackey obstruction is trivial so we can be more precise:
\begin{lemma}\label{no-obstruction}
The stepwise square integrable $\pi_\lambda$ extends to a representation
$\pi_\lambda^\dagger$ of $N\rtimes M_\lambda$ on the
representation space of $\pi_\lambda$\,.
\end{lemma}

\begin{proof}
The group $M$ preserves each $\gz_r^*$, so 
$M_\lambda = \bigcap_{\lambda_r} M_{\lambda_r}$
where $\lambda = \sum \lambda_r$ with $\lambda_r \in \gz_r^*$\,.  Recall the
construction of $\pi_\lambda$ from the decomposition $N = L_1\dots L_m$
of (\ref{setup}) and the square integrable representations $\pi_{\lambda_r}$
of the Heisenberg (or abelian) groups $L_r$ from \cite{W2013}\,.  The point is
that $\pi_{\lambda_1}$ extends to $\widetilde{\pi_{\lambda_1}} \in
\widehat{L_1L_2}$ and then we have $\pi_{\lambda_1 + \lambda_2} :=
\widetilde{\pi_{\lambda_1}} \hat\otimes \pi_{\lambda_2}$, 
$\pi_{\lambda_1 + \lambda_2}$ extends to
$\widetilde{\pi_{\lambda_1 + \lambda_2}} \in \widehat{L_1L_2L_3}$ giving
$\pi_{\lambda_1 + \lambda_2 + \lambda_3} := 
\widetilde{\pi_{\lambda_1 + \lambda_2}} \hat\otimes \pi_{\lambda_3}$, etc.
Note that we use tilde to denote extension to the next step in the
decomposition (\ref{setup}) of $N$.  
\medskip

The Fock representation of the $2n+1$ dimensional Heisenberg group $H$
extends to the semidirect product $H\rtimes U(n)$ \cite{W1975}; so each
$\pi_{\lambda_r}$ extends to $L_r\rtimes M_{\lambda_r}$\,.
We use this to modify the construction of $\pi_\lambda$ just described.
We will use dagger to denote extension from $N_*$ to $N_*\rtimes M_*$\,, 
prime to denote dagger together with tilde, and double prime to denote 
an appropriate restriction of dagger or prime.
\medskip

Let $\pi^\dagger_{\lambda_1}$ denote the extension of $\pi_{\lambda_1}$ from
$L_1$ to $L_1\rtimes M_{\lambda_1}$\,.  Now extend $\pi^\dagger_{\lambda_1}$
(instead of $\pi_{\lambda_1}$), obtaining an extension $\pi'_{\lambda_1}$
of $\pi_{\lambda_1}$ from $L_1\rtimes M_{\lambda_1}$ to 
$(L_1L_2) \rtimes M_{\lambda_1}$\,.
It restricts to a representation $\pi''_{\lambda_1}$ of
$(L_1L_2)\rtimes (M_{\lambda_1}\cap M_{\lambda_2})$.  We have the extension
$\pi^\dagger_{\lambda_2}$ of $\pi_{\lambda_2}$ from $L_2$ to 
$L_2\rtimes M_{\lambda_2}$\,; let $\pi''_{\lambda_2}$ denote its
restriction to $L_2\rtimes (M_{\lambda_1}\cap M_{\lambda_2})$.
That gives us an extension $\pi^\dagger_{\lambda_1 + \lambda_2} :=
\pi''_{\lambda_1} \hat\otimes \pi''_{\lambda_2}$ of 
$\pi_{\lambda_1 + \lambda_2}$
from $L_1L_2$ to $(L_1L_2)\rtimes (M_{\lambda_1}\cap M_{\lambda_2})$.
Continuing this way, we construct the extension of $\pi_\lambda$ from
$N$ to $N\rtimes M_\lambda$\,.
\end{proof}

\begin{remark}{\rm One can also prove Lemma \ref{no-obstruction}
by combining the Mackey obstructions $[\gamma_r] \in H^2(M_{\lambda_r};U(1))$
to extension of $\pi_{\lambda_r}$ from $N_r$ to $N_r\rtimes M_{\lambda_r}$.
In effect the cocycle $\gamma \in Z^2(M_\lambda;U(1))$ whose cohomology
class is the Mackey obstruction to extension of
$\pi_\lambda$ from $N$ to $N\rtimes M_\lambda$, is cohomologous to the
pointwise product of the $(\gamma_r)|_{M_\lambda \times M_\lambda}$,
and each $[(\gamma_r)|_{M_\lambda \times M_\lambda}] \in H^2(M_\lambda;U(1))$
is trivial because each $[\gamma_r] \in H^2(M_{\lambda_r};U(1))$ is trivial.}
\end{remark}

Each $\lambda \in \gt^*$ now defines classes
\begin{equation}\label{nm-lambda-family}
\cE(\lambda) := \left \{\pi_\lambda^\dagger \otimes \gamma \mid
\gamma \in \widehat{M_\lambda}\right \}
\text{ and } \cF(\lambda) := \left \{\Ind_{NM_\lambda}^{NM} 
(\pi_\lambda^\dagger \otimes \gamma )\mid \pi_\lambda^\dagger
\otimes \gamma \in \cE(\lambda)\right \}
\end{equation} 
of irreducible unitary representations of $N\rtimes M_\lambda$ and
$NM$.  The Mackey little group method, plus the fact that the
Plancherel density on $\widehat{N}$ is polynomial on $\gs^*$\,, 
and $\gs^*\setminus\gu^*$ has measure $0$ in $\gt^*$, 
gives us

\begin{proposition}\label{rep-mn}
Plancherel measure for $NM$ is concentrated on the set 
$\bigcup_{\lambda \in \gu^*}\cF(\lambda)$ of $($equivalence classes
of\,$)$ irreducible representations given by 
$\eta_{\lambda,\gamma} := \Ind_{NM_\lambda}^{NM} 
(\pi_\lambda^\dagger \otimes \gamma)$ with
$\pi_\lambda^\dagger \otimes \gamma \in \cE(\lambda)$ and 
$\lambda \in \gu^*$.  Further
$$
\eta_{\lambda,\gamma}|_N = 
\left . \left ( \Ind_{NM_\lambda}^{NM}(\pi_\lambda^\dagger 
	\otimes \gamma)\right ) \right |_N = \int_{M/M_\lambda} 
	(\dim \gamma)\, \pi_{\Ad^*(m)\lambda}\, d(mM_\lambda).
$$
\end{proposition}

In view of Lemma \ref{princ-orbit} there is a Borel section $\sigma$ to
$\gu^* \to \gu^*/\Ad^*(M)$ which picks out an element in each $M$-orbit
so that $M$ has the same isotropy subgroup at each of those elements.  In
other words in each $M$-orbit on $\gu^*$ we measurably choose an element
$\lambda = \sigma(\Ad^*(M)\lambda)$ such that those isotropy subgroups
$M_\lambda$ are all the same.  Let us denote
\begin{equation}\label{m-diamond}
M_\diamondsuit \text{: isotropy subgroup of } M \text{ at }
\sigma(\Ad^*(M)\lambda) \text{ for every } \lambda \in \gu^*
\end{equation}
Then we can replace $M_\lambda$ by $M_\diamondsuit$, independent of $\lambda
\in \gu^*$, in Proposition \ref{rep-mn}.  That lets us assemble to 
representations of Proposition \ref{rep-mn} for a Plancherel Formula, as 
follows.
Since $M$ is compact, we have the Schwartz space $\cC(NM)$ just as in the
discussion of $\cC(N)$ between (\ref{c-d}) and Theorem \ref{plancherel-general},
except that the pullback $\exp^*\cC(NM) \ne \cC(\gn + \gm)$.  The
same applies to $\cC(NA)$ and $\cC(NAM)$

\begin{proposition}\label{planch-mn}
Let $f \in \cC(NM)$ and write $(f_m)(n) = f(nm) = ({}_nf)(m)$ for $n \in N$
and $m \in M$.  The Plancherel density at $\Ind_{NM_\diamondsuit}^{NM} 
(\pi_\lambda^\dagger \otimes \gamma)$ is $(\dim \gamma)|\Pf(\lambda)|$ 
and the Plancherel Formula for $NM$ is
$$
f(nm) = c\int_{\gu^*/\Ad^*(M)}\, 
	\sum_{\cF(\lambda)}
	\tr \eta_{\lambda,\gamma}(_n f _m)\cdot 
	\dim(\gamma)\cdot |\Pf(\lambda)|d\lambda
$$
where $c = 2^{d_1 + \dots + d_m} d_1! d_2! \dots d_m!$\,, from 
{\rm (\ref{c-d})},
as in {\rm Theorem \ref{plancherel-general}}.
\end{proposition}

\section{The Pfaffian and the Dixmier--Puk\' anszky Operator}
\label{sec4}
\setcounter{equation}{0}
Let $Q$ be a separable locally compact group of type I.
Then \cite[\S 1]{LW1978} the Plancherel Formula for $Q$
has form 
\begin{equation}\label{LW}
f(x) = \int_{\widehat{Q}} \tr\pi(D(r(x)f)) d\mu_{_Q}(\pi)
\end{equation}
where $D$ is an invertible positive self adjoint operator on $\cL^2(Q)$,
conjugation--semi-invariant of weight equal to the modular function $\delta_Q$,
and $\mu$ is a positive Borel measure on the unitary dual $\widehat{Q}$.
The operator $D$ is needed for the following reason.  If $Q$ were unimodular
its Plancherel Formula would be of the form $f(1) = \int_{\widehat{Q}}
\tr\pi(f))d\mu_{_Q}(\pi)$ with both sides invariant under conjugation by
elements of $Q$.  In general, however, the left hand side $f(1)$ is 
conjugation--invariant while conjugation transforms $\pi(f) = \int_Q
f(x)\pi(x)dx$, and thus the the right hand side 
$\int_{\widehat{Q}}\tr\pi(f))d\mu_{_Q}(\pi)$,
by the modular function.  Thus the modular function has to be
somehow compensated, and that is the role of $D$.
If $Q$ is unimodular then $D$ is the identity and (\ref{LW}) reduces to the
usual Plancherel Formula.  The point is that semi-invariance of $D$
compensates any lack of unimodularity.  See \cite[\S 1]{LW1978} for
a detailed discussion, including a discussion of the domain of $D$
and $D^{1/2}$.
\medskip

Uniqueness of the pair $(D,\mu)$ remains unsettled, though of course
$D \otimes \mu$ is unique (up to normalization of Haar measures), so
one tries to find a ``best'' choice of $D$.  Given any such pair $(D,\mu)$
we refer to $D$ as a {\sl Dixmier--Puk\' anszky Operator} on $Q$ and
to $\mu$ as the associated {\sl Plancherel measure} on $\widehat{Q}$.
\medskip

In this section we exhibit an explicit Dixmier--Puk\' anszky Operator
for the minimal parabolic $P = MAN$ and its solvable subgroup $AN$.
Those groups are never unimodular.  Our Dixmier--Puk\' anszky Operator
is constructed from the Pfaffian polynomial $\Pf(\lambda)$ and a
certain ``quasi-central determinant'' function on $\gs^*$.
\medskip

Let $\delta$ denote the modular function on $P = MAN$.  As $M$ is compact
and $\Ad_P(N)$ is unipotent on $\gp$, $MN$ is in the kernel of $\delta$.  
So $\delta$ is determined by its values on $A$, where it is given by
$\delta(\exp(\xi)) = \exp(\tr(\ad(\xi)))$.  There $\xi = \log a \in \ga$.

\begin{lemma}\label{trace}
Let $\xi \in \ga$.  Then $\frac{1}{2}(\dim \gl_r + \dim \gz_r) \in \Z$
for $1\leqq r\leqq m$ and\\
\phantom{ii}{\rm (i)} the trace of $\ad(\xi)$ on $\gl_r$ is
  $\frac{1}{2}(\dim \gl_r + \dim \gz_r)\beta_r(\xi)$, \\
\phantom{i}{\rm (ii)} the trace of $\ad(\xi)$ on $\gn$ and on $\gp$ is
  $\frac{1}{2}\sum_r (\dim \gl_r + \dim \gz_r)\beta_r(\xi)$, and \\
{\rm (iii)}
  the determinant of $\Ad(\exp(\xi))$ on $\gn$ and on $\gp$ is
  $\prod_r \exp(\beta_r(\xi))^{\frac{1}{2} (\dim \gl_r + \dim \gz_r)}$.
\end{lemma}

\begin{proof} Decompose $\gl_r = \gz_r + \gv_r$ where $\gz_r = \gg_{\beta_r}$
is its center and $\gv_r = \sum_{\alpha \in \Delta_r^+}\gg_\alpha$\,.
The set $\Delta_r^+$ is the disjoint union of sets 
$\{\alpha, \beta_r - \alpha\}$ and (if $\frac{1}{2}\beta_r$ is a root)
$\{\frac{1}{2}\beta_r\}$.  That proves the integrality assertion. 
From (\ref{beta-reflect}) and 
Lemma \ref{layers-nilpotent} we have
$\dim \gg_\alpha = \dim\gg_{\beta_r - \alpha}$.
So the trace of $\ad(\xi)$ on $\gv_r$ adds
up to $\frac{1}{2}(\dim \gv_r)\beta_r(\xi)$.  On $\gz_r = \gg_{\beta_r}$
it is of course $(\dim \gz_r) \beta_r(\xi)$.  That proves (i).  For (ii)
we take the sum over $\{\beta_1, \dots , \beta_m\}$ and then for (iii)
we exponentiate.
\end{proof}

Since $\delta = \det\Ad$, Lemma \ref{trace}(iii) can be formulated as
\begin{lemma}\label{modular}
The modular function $\delta = \delta_P$ of $P = MAN$ is $\delta(man) 
= \prod_r \exp(\beta_r(\log a))^{\frac{1}{2} (\dim \gl_r + \dim \gz_r)}$.
The modular function $\delta_{AN}$ of $AN$ is $\delta_P|_{AN}$.
\end{lemma}

We  consider semi-invariance of the Pfaffian.  Let $\xi \in \ga$ and
consider a basis $\{x_i\}$ of $\gv_r$, each element in some $\gg_\alpha$ with
$\alpha \in \Delta_r^+$, in which $b_\lambda$ has matrix consisting of $2\times 2$
blocks $\left ( \begin{smallmatrix} 0 & 1 \\ -1 & 0 \end{smallmatrix}\right )$
down the diagonal.  But $-\ad^*(\xi)(\lambda)[x_i,x_j] = 
\lambda(\ad(\xi)[x_i,x_j]) = \lambda[\ad(\xi)x_i,x_j] + \lambda([x_i,\ad(\xi)x_j]
= \beta_r(\xi)\lambda([x_i,x_j])$ as in the proof of Lemma \ref{trace}.
Now $(\ad(\xi)\Pf)|_{\gv_r}(\lambda) = \Pf|_{\gv_r}(-\ad^*(\xi)(\lambda))
= \frac{1}{2}\dim \gv_r\,\beta_r (\xi)\Pf|_{\gv_r}$\,.  Sum over $r$:

\begin{lemma}\label{trace2}
Let $\xi \in \ga$ and $a = \exp(\xi) \in A$.  Then 
$\ad(\xi)\Pf = \left (\frac{1}{2}\sum_r\, \dim (\gl_r/\gz_r) \beta_r(\xi)\right )\Pf$
and $\Ad(a)\Pf = 
\left ( \prod_r \exp(\beta_r(\xi))^{\frac{1}{2}\dim (\gl_r / \dim \gz_r)}\right )\Pf$.
\end{lemma}

At this point it is convenient to introduce some notation and definitions.
\begin{definition} {\rm The algebra $\gs$ is the {\sl quasi-center} of 
$\gn$.  Then 
$\Det_{\gs^*}(\lambda):=\prod_r (\beta_r(\lambda))^{\dim \gg_{\beta_r}}$
is a polynomial function on $\gs^*$, the {\sl quasi-center determinant}. 
}
\end{definition}
If $\xi \in \ga$ and $a = \exp(\xi) \in A$ we compute 
\begin{equation}\label{trace3}
\begin{aligned}
(\Ad(a)\Det_{\gs^*})(\lambda) &= \Det_{\gs^*}(\Ad^*(a^{-1})(\lambda)) \\ 
&= {\prod}_r (\beta_r(\Ad(a^{-1})^*\lambda))^{\dim \gg_{\beta_r}}
= {\prod}_r (\beta_r(\exp(\beta_r(\xi))\lambda))^{\dim \gg_{\beta_r}}.
\end{aligned}
\end{equation}

Combining Lemmas \ref{trace} and \ref{trace2} with (\ref{trace3}) we have

\begin{proposition}\label{quasi-invariance}
The product $\Pf\cdot\Det_{\gs^*}$ is an $\Ad(MAN)$-semi-invariant
polynomial on $\gs^*$ of degree 
$\frac{1}{2}(\dim \gn + \dim \gs)$ and of 
weight equal to the modular function $\delta_{MAN}$\,.
\end{proposition}

Our fixed decomposition $\gn = \gv + \gs$ gives $N = VS$ where
$V = \exp(\gv)$ and $S = \exp(\gs)$.  Now define
\begin{equation}\label{defdp}
D: \text{ Fourier transform of } \Pf\cdot\Det_{\gs^*} 
	\text{, acting on $MAN = MAVS$  by acting on the  $S$  variable.}
\end{equation}
We use the fact that the definition of $\cC(N)$ between
(\ref{c-d}) and Theorem \ref{plancherel-general} applies to $\cC(MAN)$:

\begin{theorem}\label{dp-min-parab}
The operator $D$ of {\rm (\ref{defdp})} is an invertible 
self-adjoint differential operator of degree 
$\frac{1}{2}(\dim \gn + \dim \gs)$ on $\cL^2(MAN)$ with dense
domain $\cC(MAN)$, and it is $\Ad(MAN)$-semi-invariant of
weight equal to the modular function $\delta_{MAN}$\,.  
In other words $|D|$ is a Dixmier--Puk\' anszky Operator
on $MAN$ with domain equal to the space of rapidly decreasing $C^\infty$
functions.
\end{theorem}

\begin{proof}
Since it is the Fourier transform of a real polynomial, $D$ is a 
differential operator which is invertible and self-adjoint on 
$\cL^2(MAN)$.  Its degree as a differential operator is the
same as that of the polynomial.  Further, it has dense
domain $\cC(MAN)$.  Proposition \ref{quasi-invariance} ensures that 
the degree is $\frac{1}{2}(\dim \gn + \dim \gs)$ and that $D$ is 
$\Ad(MAN)$--semi--invariant as asserted.
\end{proof}

\section{Generic Representations}
\label{sec5}
\setcounter{equation}{0}
In this section we complete the description of a dense open subset 
of the unitary dual of $\widehat{P} = \widehat{MAN}$ that carries
Plancherel measure.  In the next section we will combine this with 
Theorem \ref{dp-min-parab}, using the framework of (\ref{LW}), to 
obtain explicit Plancherel Formulae for $MAN$ and $AN$.
\medskip

There are two paths here.  We can obtain the generic representations
of $P$ by inducing the representations $\Ind_{NM_\lambda}^{NM} 
\eta_{\lambda,\gamma}$ discussed in Proposition \ref{rep-mn}.  But one
has a cleaner final statement if he avoids that induction by stages 
and induces directly from $N\rtimes (MA)_\lambda$ to $P$.
\medskip

Since $\lambda \in \gt^*$ has nonzero projection on each summand $\gz_r^*$
of $\gs^*$, and $a \in A$ acts by the positive real scalar
$\exp(\beta_r(\log(a)))$ on $\gz_r$, 
\begin{equation}\label{alambda}
A_\lambda = \exp(\{\xi \in \ga \mid \text{each } \beta_r(\xi) = 0\}),
\text{ independent of } \lambda \in \gt^*.
\end{equation}
Because of this independence, and in view of our earlier definition of
$\ga_\diamondsuit = \{\xi \in \ga \mid \text{ each } \beta_r(\xi) = 0\}$,
we define
\begin{equation}\label{adiamond}
A_\diamondsuit = A_\lambda 
  \text{ for any (and thus for all) } \lambda \in \gt^*.
\end{equation}

\begin{lemma}\label{ma-diamond}
In the notation of {\rm (\ref{m-diamond})} and {\rm (\ref{adiamond})},
if $\lambda \in \sigma(\gu^*)$ then the stabilizer
$(MA)_\lambda = M_\diamondsuit A_\diamondsuit$\,.
\end{lemma}

\begin{proof} As $\lambda \in \gt^*$ it has expression $\lambda =
\sum \lambda_r$ with $0 \ne \lambda_r \in \gz^* = \gg_{\beta_r}$\,.
Let $\xi \in \ga$ and $m \in M$ with $\Ad^*(\exp(\xi)m)\lambda = \lambda$.
Then each $\Ad^*(\exp(\xi)m)\lambda_r = \lambda_r$\,.  In an
$\Ad^*(M)$-invariant inner product, $||\Ad^*(\exp(\xi)m)\lambda_r|| =
\exp(\beta_r(\xi))||\lambda_r||$ so each $\beta_r(\xi) = 0$, i.e.
$\xi \in \ga_\diamondsuit$ and $\Ad^*(\exp(\xi)m)\lambda =
\Ad^*(m)\lambda$.  Thus $m \in M_\diamondsuit$ and $\exp(\xi) \in
A_\diamondsuit$\,, as asserted.
\end{proof}

Now we are ready to use the Mackey little group method.  First, there is no
problem with obstructions:

\begin{lemma}\label{no-ma-obstruction}
Let $\lambda \in \sigma(\gu^*)$ and note the extension 
$\pi_\lambda^\dagger$ of $\pi_\lambda$ from $N$ to $NM_\diamondsuit$
defined by {\rm Lemma \ref{no-obstruction}}. Then $\pi_\lambda^\dagger$
extends further to a unitary representation 
$\widetilde{\pi_\lambda}$ of $NM_\diamondsuit A_\diamondsuit$
on the representation space of $\pi_\lambda$\,.
\end{lemma}

\begin{proof}  Since $A_\diamondsuit$ is a vector group, it 
retracts to a point, so $H^2(A_\diamondsuit;U(1)) 
= H^2(\text{point};U(1)) = \{1\}$.  Thus
the Mackey obstruction vanishes.  
\end{proof}

Let $\lambda \in \sigma(\gu^*)$.  Note that
$\widehat{A_\diamondsuit}$ consists of the unitary characters 
$\exp(i\phi): a \mapsto e^{i\phi(\log a)}$ with $\phi \in \ga_\diamondsuit^*$.
With that notation, the representations of $P$
corresponding to $\lambda$ are the

\begin{equation}\label{lambda-family}
\pi_{\lambda,\gamma,\phi} := \Ind_{NM_\diamondsuit A_\diamondsuit}^{NMA}
  (\widetilde{\pi_\lambda} \otimes \gamma \otimes \exp(i \phi))
  \text{ where } \gamma \in \widehat{M_\diamondsuit} \text{ and }
  \phi \in \ga^*_\diamondsuit\,.
\end{equation}
Here the action of $A$ fixes $\gamma$ because $A$ centralizes $M$, and it
fixes $\phi$ because $A$ is commutative, so
\begin{equation} \label{man-equiv}
\pi_{\lambda,\gamma,\phi}\cdot \Ad((ma)^{-1})
 = \pi_{\Ad^*(ma)\lambda, \gamma, \phi}
\end{equation}

\begin{proposition}\label{rep-man}
Plancherel measure for $MAN$ is concentrated on the set 
of unitary equivalence classes of representations
$\pi_{\lambda,\gamma,\phi}$ for $\lambda \in \sigma(\gu^*)$, 
$\gamma \in \widehat{M_\diamondsuit}$\, and
   $\phi \in \ga_\diamondsuit^*$\,.  The equivalence
class of $\pi_{\lambda,\gamma,\phi}$ depends only on
$(\Ad^*(MA)\lambda, \gamma, \phi)$.
\end{proposition}

Representations of $AN$ are the case $\gamma = 1$.  In effect, let
$\pi'_\lambda$ denote the obvious extension $\widetilde{\pi_\lambda}|_{AN}$
of the stepwise square integrable
representation $\pi_\lambda$ from $N$ to $NA_\diamondsuit$ 
where $\widetilde{\pi_\lambda}$ is given by Lemma \ref{no-ma-obstruction}.
Denote
\begin{equation}\label{irr-na}
	\pi_{\lambda,\phi} = \Ind_{NA_\diamondsuit}^{NA}
	(\pi'_\lambda \otimes \exp(i \phi)) \text{ where }
	\lambda \in \gu^* \text{ and } \phi \in \ga_\diamondsuit^*.
\end{equation}
Then 
$\pi_{\lambda,\phi}$ and $\pi_{\lambda',\phi}$ are equivalent if
and only if $\lambda' \in \Ad^*(A)\lambda$.  We have proved

\begin{corollary}\label{rep-ma}
Plancherel measure for $AN$ is concentrated on the set
$\{\pi_{\lambda,\phi} \mid  \lambda \in \gu^* \text{ and } 
\phi \in \ga_\diamondsuit^*\}$
of $($equivalence classes of\,$)$ irreducible representations of
$AN = NA$ described in {\rm (\ref{irr-na})}.
\end{corollary}

Finally we describe the set $\Ad^*(MA)\lambda$ of Proposition \ref{rep-man}.
A result of C.C. Moore says that $\Ad(P_\C)$ has a Zariski open orbit
on $\gn_{_\C}^*$, so there is a finite set of open $\Ad(P)$-orbits
on $\widehat{N}$ such that Plancherel measure is concentrated on the
union of those open orbits.  Moore presented this and a number of related
results in a January 1972 seminar at Berkeley but he didn't publish it.
Carmona circulated a variation on this but he also seems to have left it
unpublished.  Using Lemma \ref{ma-diamond}, Moore's result leads directly to
\begin{lemma}\label{finite-ma-orbit}
The $\Pf$-nonsingular principal orbit set $\gu^*$ is a finite union of
open $\Ad^*(MA)$-orbits.
\end{lemma}

Let $\{\cO_1\,, \dots \cO_v\}$ denote the (open) $\Ad^*(MA)$-orbits on
$\gu^*$.  Denote
\begin{equation}
\lambda_i = \sigma(\cO_i) \text{\phantom{X}so\phantom{X}} 
  \cO_i = \Ad^*(MA)\lambda_i \text{\phantom{X}and\phantom{X}} 
  (MA)_{\lambda_i} = M_\diamondsuit A_\diamondsuit
  \text{\phantom{X}for\phantom{X}} 1 \leqq i \leqq v.
\end{equation}
Then Proposition \ref{rep-man} becomes
\begin{theorem}\label{rep-man-ref}
Plancherel measure for $MAN$ is concentrated on 
the set $($of equivalence classes of\,$)$
unitary representations
$\pi_{\lambda_i,\gamma,\phi}$ for $1 \leqq i \leqq v$,
$\gamma \in \widehat{M_\diamondsuit}$\, and
   $\phi \in \ga_\diamondsuit^*$\,.  
\end{theorem}

\section{Non--Unimodular Plancherel Formulae}
\label{sec6}
\setcounter{equation}{0}
Recall the Dixmier--Puk\' ansky operator $D$ from (\ref{defdp}) and
Theorem \ref{dp-min-parab}.  The Plancherel Formula (or Fourier inversion
formula) for $MAN$ is

\begin{theorem}\label{planch-man}
Let $P = MAN$ be a minimal parabolic subgroup of the real reductive
Lie group $G$.  Given $\pi_{\lambda,\gamma,\phi} \in \widehat{MAN}$ as 
described in {\rm (\ref{lambda-family})} let
$\Theta_{\pi_{\lambda,\gamma,\phi}}: h \mapsto 
    \tr \pi_{\lambda,\gamma,\phi}(h)$ 
denote its distribution character.  Then
$\Theta_{\pi_{\lambda,\gamma,\phi}}$ is a tempered distribution.  
If $f \in \cC(MAN)$ then 
$$
f(x) = c\sum_{i=1}^{v} \sum_{\gamma \in \widehat{M_\diamondsuit}}
	\int_{\ga^*_\diamondsuit} 
	\Theta_{\pi_{\lambda_i,\gamma,\phi}}(D(r(x)f)) |\Pf(\lambda_i)| 
	\dim \gamma\,\,d\phi
$$
where $c > 0$ depends on normalizations of Haar measures.
\end{theorem}

\begin{proof}  We compute along the lines of the argument of 
\cite[Theorem 2.7]{LW1982}, ignoring multiplicative constants that
depend of normalizations of Haar measures.  From \cite[Theorem 3.2]{KL1972},
$$
\begin{aligned}
\tr &\pi_{\lambda_i,\gamma,\phi}(Dh) \\
	&=\int_{x \in MA/M_\diamondsuit A_\diamondsuit}\delta(x)^{-1}
		\tr \int_{NM_\diamondsuit A_\diamondsuit}(Dh)(x^{-1}nmax)\cdot
		(\pi_{\lambda_i}\otimes \gamma \otimes \exp(i\phi))
		   (nma)\,dn\, dm\, da\, dx\\
&= \int_{x \in MA/M_\diamondsuit A_\diamondsuit} 
	\tr \int_{NM_\diamondsuit A_\diamondsuit}(Dh)(nx^{-1}max)\cdot
		(\pi_{\lambda_i}\otimes \gamma \otimes \exp(i\phi))
		(xnx^{-1}ma)\,
		dn\, dm\, da\, dx.
\end{aligned}
$$
Now
\begin{equation}\label{one-orbit}
\begin{aligned}
\int_{\widehat{M_\diamondsuit A_\diamondsuit}} 
  &\tr \pi_{\lambda_i,\gamma,\phi}(Dh) \dim\gamma\,\,d\phi \\
&= \int_{\widehat{M_\diamondsuit A_\diamondsuit}} 
  \int_{x \in MA/M_\diamondsuit A_\diamondsuit} 
        \tr \int_{NM_\diamondsuit A_\diamondsuit}(Dh)(nx^{-1}max)\times \\
	&\phantom{XXXXXXXXXXXXXXX}\times
                (\pi_{\lambda_i}\otimes \gamma \otimes \exp(i\phi))
                (xnx^{-1}ma)\,
                dn\, dm\, da\, dx\, \dim\gamma\,\,d\phi\\
&= \int_{x\in MA/M_\diamondsuit A_\diamondsuit} 
	\int_{\widehat{M_\diamondsuit A_\diamondsuit}} 
	\tr \int_{NM_\diamondsuit A_\diamondsuit}
        (Dh)(nx^{-1}max) \times \\
        &\phantom{XXXXXXXXXXXXXXX}\times
		(\pi_{\lambda_i}\otimes \gamma\otimes \exp(i\phi)
        	(xnx^{-1}ma)\, dn\, dm\, da\, \dim\gamma\,\,d\phi\,\, dx \\
&= \int_{x \in MA/M_\diamondsuit A_\diamondsuit} \tr \int_N (Dh)(n) 
	\pi_{\lambda_i}(xnx^{-1}) dn\, dx \\
&= \int_{x \in MA/M_\diamondsuit A_\diamondsuit} \tr \int_N (Dh)(n) 
        (x^{-1}\cdot \pi_{\lambda_i})(n) dn\, dx \\
&= \int_{x \in MA/M_\diamondsuit A_\diamondsuit} \tr ((x^{-1}\cdot
	\pi_{\lambda_i})(Dh))\, dx\\
&= \int_{x \in MA/M_\diamondsuit A_\diamondsuit} 
	(x^{-1}\cdot\pi_{\lambda_i})_*(D)
	\,\,\tr (x^{-1}\cdot\pi_{\lambda_i})(h) dx\\
&= \int_{x \in MA/M_\diamondsuit A_\diamondsuit}(\pi_{\lambda_i})_*(x\cdot D)
	\,\tr (x^{-1}\cdot\pi_{\lambda_i})(h)\, dx \\
&= \int_{x \in MA/M_\diamondsuit A_\diamondsuit} \delta_{MAN}(x)
	\,\tr (x^{-1}\cdot\pi_{\lambda_i})(h)\, dx 
= \int_{\Ad^*(MA)\lambda_i} \tr \pi_\lambda(h) |\Pf(\lambda)|d\lambda .
\end{aligned}
\end{equation}
Summing over the orbits $\cO_i$ of $\Ad^*(MA)$ on $\gu^*$ we now have
\begin{equation}\label{all-orbits}
\begin{aligned}
\sum_{i=1}^v &\sum_{\gamma \in \widehat{M_\diamondsuit}}
	\int_{\ga^*_\diamondsuit}
	\tr \pi_{\lambda_i,\gamma,\phi}(Dh) \dim\gamma\,\,d\phi 
= \sum_{i=1}^v \int_{\widehat{M_\diamondsuit A_\diamondsuit}} 
  \tr \pi_{\lambda_i,\gamma,\phi}(Dh) \dim\gamma\,\,d\phi \\
& = \sum_{i=1}^v \int_{\cO_i} \tr \pi_\lambda(h) 
		|\Pf(\lambda)|d\lambda 
= \int_{\gu^*} \tr \pi_\lambda(h) |\Pf(\lambda)|d\lambda = h(1_N) = h(1_P)\,.
\end{aligned}
\end{equation}
Let $h$ denote any right translate of $f$.  The theorem follows.
\end{proof}

The Plancherel Theorem for $NA$ follows similar lines. 
For the main computation (\ref{one-orbit}) in Theorem \ref{planch-man} we 
omit $M$ and $\gamma$.  That gives
\begin{equation}\label{one-a-orbit}
\int_{\ga_\diamondsuit^*} 
  \tr \pi_{\lambda_0,\phi}(Dh) \,d\phi 
= \int_{\Ad^*(A)\lambda_0} \tr \pi_\lambda(h) |\Pf(\lambda)|d\lambda 
\end{equation}
In order to go from an $\Ad^*(A)\lambda_0$ in (\ref{one-a-orbit}) to
an integral over $\gu^*$ we use $M$ to parameterize the space of
$\Ad^*(A)$-orbits on $\gu^*$.  
We first note that
\begin{equation}\label{no-intersect}
\text{If } \lambda \in \gu^* \text{ then } 
	\Ad^*(A)\lambda \cap \Ad^*(M)\lambda = \{\lambda\}
\end{equation}
because $\Ad^*(A)$ acts on each $\gz^*_r$ be positive scalars and 
$\Ad^*(M)$ preserves the norm on each $\gz^*_r$\,.  Thus the space of
$\Ad^*(A)$-orbits on $\gu^*$ is partitioned by the space of
$\Ad^*(M)$-orbits on $\gu^*/\Ad^*(A)$.  Each such $\Ad^*(M)$-orbit
is in fact an $\Ad^*(MA)$-orbit on $\gu^*$\,.  Recall the
decomposition $\gu^* = \bigcup \cO_i$ where $\cO_i = \Ad^*(MA)\lambda_i$
with $\lambda_i = \sigma(\Ad^*(M)\lambda_i)$. Define $S_i = 
\Ad^*(M)\lambda_i$\,, so $\gu^* = \bigcup_i \Ad^*(A)S_i$\,. Now

\begin{proposition}\label{na-reps}
Plancherel measure for $NA$ is concentrated on the equivalence classes
of representations $\pi_{\lambda,\phi} 
= \Ind_{NA_\diamondsuit}^{NA}(\pi'_{\lambda}\otimes \exp(i\phi))$ where
$\lambda \in S_i := \Ad^*(M)\lambda_i\,\, (1 \leqq i \leqq v)$, 
$\pi'_{\lambda}$ is the extension of 
$\pi_\lambda$ from $N$ to $NA_\diamond$ and $\phi \in \ga_\diamond^*$\,.  
Representations $\pi_{\lambda,\phi}$ and $\pi_{\lambda',\phi'}$ are
equivalent if and only if $\lambda' \in \Ad^*(A)\lambda$ and
$\phi' = \phi$.  Further, $\pi_{\lambda,\phi}|_N = 
\int_{a \in A/A_\diamondsuit} \pi_{\Ad^*(a)\lambda}da$.
\end{proposition}
Now we sum both sides of (\ref{one-a-orbit}) as follows.
\begin{equation}
\begin{aligned}
\sum_i\, \int_{\lambda' \in S_i}\int_{\ga_\diamondsuit^*} 
  \tr &\pi_{\lambda',\phi}(Dh) \,d\phi\, d\lambda'
= \sum_i \int_{\cO_i} \tr \pi_\lambda(h) |\Pf(\lambda)|d\lambda\\
&= \int_{\gu^*} \tr \pi_\lambda(h) |\Pf(\lambda)|d\lambda = h(1_N) = h(1_{AN}).
\end{aligned}
\end{equation}
Again taking $h = r(x)f$ we have

\begin{theorem}\label{planch-an}
Let $P = MAN$ be a minimal parabolic subgroup of the real reductive
Lie group $G$.  Given $\pi_{\lambda,\phi} \in \widehat{AN}$ as
described in {\rm Proposition \ref{na-reps}} let
$\Theta_{\pi_{\lambda, \phi}}: h \mapsto 
	\tr \pi_{\lambda,\phi}(h)$
denote its distribution character.  Then
$\Theta_{\pi_{\lambda, \phi}}$ is a tempered distribution.
If $f \in \cC(AN)$ then
$$
f(x) = c\sum_{i=1}^v \int_{\lambda \in S_i} 
	\int_{\ga^*_\diamondsuit}
        \tr \pi_{\lambda,\phi}(D(r(x)f)) |\Pf(\lambda)|
	d\lambda d\phi.
$$
where $c = 2^{d_1 + \dots + d_m} d_1! d_2! \dots d_m!$\,, from
{\rm (\ref{c-d})},
as in {\rm Theorem \ref{plancherel-general}} and 
{\rm Proposition \ref{planch-mn}}.
\end{theorem}

\section{Remark on Strongly Orthogonal Restricted Roots}
\label{sec7}
\setcounter{equation}{0}

The goal of this paper was to extend our earlier result, 
Theorem \ref{iwasawa-layers},
from nilradicals of minimal parabolic subgroups to the minimal parabolics 
themselves.  In part we needed to extend some results of Kostant
(\cite{K2011}, \cite{K2012}) on strongly orthogonal roots 
from Borel subalgebras of complex semisimple Lie 
algebras to minimal parabolic subalgebras of real semisimple algebras.  But 
some of the technical results in (\cite{K2011}, \cite{K2012}), which we didn't
use but are of strong independent interest, also extend.  We use the notation 
of Section \ref{sec2}.

\begin{lemma} \label{neg-no-prob}
$\Delta_r^+ = \{\alpha \in 
	\Delta^+(\gg,\ga) \cup -\Delta^+(\gg,\ga) \mid 
	\alpha \perp \beta_i \text{ for } i < r 
	\text{ and } \langle \alpha, \beta_r \rangle > 0\}$.
\end{lemma}

\begin{proof}
In view of (\ref{layers}) we need only show that if 
$\alpha \in -\Delta^+(\gg,\ga)$ and $\alpha \perp \beta_i$ for $i < r$
then $\langle \alpha, \beta_r\rangle \leqq 0$.  But if that fails, so
$\langle \alpha, \beta_r \rangle > 0$, then $\beta_r - \alpha$ is a root
greater than $\beta_r$ and $\perp \beta_i$ for $i < r$, which
contradicts the construction (\ref{cascade}) of the cascade of strongly
orthogonal roots $\beta_j$\,.
\end{proof}

\begin{proposition} \label{signs}
The composition $s_{\beta_1}s_{\beta_2}\dots s_{\beta_r}$
sends $(\Delta_1^+ \cup \dots \cup \Delta_r^+)$ to 
$-(\Delta_1^+ \cup \dots \cup \Delta_r^+)$.  In particular,
the longest element of the restricted Weyl group $W = W(\gg,\ga,\Delta^+)$,
defined by $w_0(\Delta^+(\gg,\ga)) = -\Delta^+(\gg,\ga)$, is given by
$w_0 = s_{\beta_1}s_{\beta_2}\dots s_{\beta_m}$.
\end{proposition}

\begin{proof}  This is an induction on $r$.  For $r = 1$ the statement is
in the discussion immediately preceding Lemma \ref{layers-nilpotent}.
Now suppose that $s_{\beta_1}s_{\beta_2}\dots s_{\beta_{r-1}}$
sends $(\Delta_1^+ \cup \dots \cup \Delta_{r-1}^+)$ to its negative.
Since $s_{\beta_r}(\beta_i) = \beta_i$ for $i < r$,
Lemma \ref{neg-no-prob} shows that $s_{\beta_r}$ preserves 
$(\Delta_1^+ \cup \dots \cup \Delta_{r-1}^+)$, so 
$s_{\beta_1}s_{\beta_2}\dots s_{\beta_r}$ sends 
$(\Delta_1^+ \cup \dots \cup \Delta_{r-1}^+)$ to its negative.  But 
Lemma \ref{neg-no-prob} also shows that 
$s_{\beta_1}s_{\beta_2}\dots s_{\beta_{r-1}}$ preserves $\Delta_r^+$,
and the discussion just before Lemma \ref{layers-nilpotent} shows that
$s_{\beta_r}$ sends $\Delta_r^+$ to its negative.  This completes the
induction.
\medskip

In view of Lemma \ref{fill-out}, the case $r = m$ says that
$s_{\beta_1}s_{\beta_2}\dots s_{\beta_m} \Delta^+(\gg,\ga) = 
-\Delta^+(\gg,\ga)$.
\end{proof}

\begin{corollary} Let $\nu \in \ga^*$ be the highest weight of an
irreducible finite dimensional representation $\tau_\nu$ of $\gg$, 
so the dual representation $\tau_\nu^*$ has highest weight
$\nu^* := -w_0(\nu)$.  Then $\nu + \nu^* = \sum 
\frac{2\langle \nu,\beta_i\rangle}{\langle\beta_i , \beta_i\rangle}\beta_i\,,$ 
integral linear combination of $\beta_1, \dots , \beta_m$.
\end{corollary}

\begin{proof} Write $(\alpha,\gamma) = 
\frac{2\langle \alpha, \gamma\rangle}{\langle \gamma,\gamma\rangle}$. 
Compute $s_{\beta_1}(\nu) = \nu - (\nu,\beta_1)\beta_1$, then
$s_{\beta_2}s_{\beta_1}(\nu) = \nu - (\nu,\beta_1)\beta_1
- (\nu,\beta_2)\beta_2$, continuing on to $s_{\beta_m}s_{\beta_{m-1}}\dots
s_{\beta_1}(\nu) = \nu - \sum (\nu,\beta_i)\beta_i$.  Using the last
statement of Proposition \ref{signs} now $\nu + \nu^* = \nu - w_0(\nu)
= \sum (\nu,\beta_i)\beta_i$ as asserted.
\end{proof}

\medskip
\noindent Department of Mathematics, University of California,\hfill\newline
\noindent Berkeley, California 94720--3840, USA\hfill\newline
\smallskip
\noindent {\tt jawolf@math.berkeley.edu}

\enddocument
\end